\documentclass[reqno]{amsart}

\usepackage{amsfonts}
\usepackage{amssymb}
\usepackage{amsmath}
\usepackage{amscd}

\newtheorem{theo}{Theorem}[section]
\newtheorem{coro}{Corollary}[section]
\newtheorem{lem}{Lemma}[section]

\theoremstyle{definition}
\newtheorem{exmp}{Example}[section]

\theoremstyle{remark}
\newtheorem{rem}{Remark}[section]

\numberwithin{equation}{section}

\begin{document}

\title[Oscillations of differential equations]{Oscillation criteria for
differential equations with several retarded arguments}

\author[G. Infante]{G. Infante$^\diamond$}
\address{Dipartimento di Matematica ed Informatica\\
Universit\`{a} della Calabria\\
87036 Arcavacata di Rende, Cosenza, Italy}
\email{gennaro.infante@unical.it}

\author[R. Koplatadze]{R. Koplatadze$^*$}
\address{Department of Mathematics\\
Tbilisi State University\\
Tbilisi, Georgia}
\email{r\_koplatadze@yahoo.com}

\author[I. P. Stavroulakis]{I. P. Stavroulakis}
\address{Department of Mathematics\\
University of Ioannina\\
451 10 \ Ioannina, Greece}
\email{ipstav@cc.uoi.gr}


\begin{abstract}
Consider the first-order linear differential equation with several retarded
arguments%
\begin{equation*}
x^{\prime }(t)+\sum\limits_{i=1}^{m}p_{i}(t)x(\tau _{i}(t))=0,\;\;\;t\geq
t_{0},
\end{equation*}%
where the functions $p_{i},\tau _{i}\in C([t_{0,}\infty ),\mathbb{R}^{+}),$
for every $i=1,2,\ldots ,m,$ $\tau _{i}(t)\leq t$ \ for $t\geq t_{0}$ and $%
\lim_{t\rightarrow \infty }\tau _{i}(t)=\infty $. 
In this paper the state-of-the-art on
the oscillation of all solutions to these equations is reviewed and 
new sufficient conditions
for the oscillation are established, especially in the case
of nonmonotone arguments. Examples illustrating the results are given.
\end{abstract}

\maketitle

\section{Introduction}

Consider the differential equation with several non-monotone retarded
arguments%
\begin{equation}
x^{\prime }(t)+\sum_{i=1}^{m}p_{i}(t)x(\tau _{i}(t))=0. \quad t\geq t_{0},%
\end{equation}%
where the functions $p_{i},\tau _{i}\in C([t_{0,}\infty ),\mathbb{R}^{+}),$
for every $i=1,2,\ldots ,m$, $($here $\mathbb{R}^{+}=[0,\infty )),$ $\tau
_{i}(t)\leq t$ \ for $t\geq t_{0}$ and $\lim\limits_{t\rightarrow \infty
}\tau_{i}(t)=\infty $.

Let $T_0\in[t_0,+\infty)$, $\tau(t)=\min\{\tau_i(t):i=1,\dots,m\}$ and $%
\tau_{(-1)}(t)=\sup\{s:\tau(s)\le t\}$. By a solution of the equation (1.1)
we understand a function $u\in C([T_0,+\infty);R)$, continuously
differentiable on $[\tau_{(-1)}(T_0),+\infty)$ and that satisfies (1.1) for $%
t\ge\tau_{(-1)}(T_0)$. Such a solution is called \textit{oscillatory} if it
has arbitrarily large zeros, and otherwise it is called \textit{%
nonoscillatory}.

In the special case where $m=1$ equation (1.1) reduces to the equation%
\begin{equation}
x^{\prime }(t)+p(t)x(\tau (t))=0,\;\;\;t\geq t_{0},
\end{equation}%
where the functions $p,\tau \in C([t_{0,}\infty ),\mathbb{R}^{+})$, $\tau
(t)\leq t$ \ for $t\geq t_{0}$ and $\lim_{t\rightarrow \infty }\tau
(t)=\infty .$

For the general theory of these equations the reader is referred to
\cite{6, 9, 11, 12, 23}.

The problem of establishing sufficient conditions for the oscillation of all
solutions to the differential equation (1.2) has been the subject of many
investigations. See, for example, \cite{1, 2, 3, 4, 5, 6, 7, 14, 15, 16, 17, 18, 19, 20, 21, 24, 25, 26, 27, 28, 29}
and the references cited therein.

In this paper we present the state-of-the-art on the oscillation of all
solutions to the equation $(1.1)$ in the case of several non-monotone
arguments and especially when the well known oscillation conditions (for $m=1)$%
\begin{equation*}
\underset{t\rightarrow \infty }{\lim \sup }\int\nolimits_{\tau
(t)}^{t}p(s)ds>1\text{ \ \ \ and \ \ }\liminf_{t\rightarrow \infty
}\int_{\tau (t)}^{t}p(s)ds>\frac{1}{e}\,.
\end{equation*}%
are not satisfied.

\section{Oscillation Criteria for Equation (1.1)}

The first systematic study for the oscillation of all solutions to the
equation $(1.2)$ was made by Myshkis. In 1950 \cite{25} he proved that every
solution oscillates if
\begin{equation*}
\limsup_{t\rightarrow \infty }[t-\tau (t)]<\infty \;\;\;\mathrm{{and}%
\;\;\;\liminf_{t\rightarrow \infty }[t-\tau (t)]\liminf_{t\rightarrow
\infty} p(t)>\frac{1}{e}\,.}
\end{equation*}

In 1972, Ladas, Lakshmikantham and Papadakis \cite{21} proved that the same
conclusion holds if in addition $\tau $ is a non-decreasing function and%
\begin{equation}
A:=\limsup_{t\rightarrow \infty }\int_{\tau (t)}^{t}p(s)ds>1.
\end{equation}

\vskip0.2cm In 1979, Ladas \cite{20} established integral conditions for the
oscillation of the equation $(1.2)$ with constant delay, while in 1982,
Koplatadze and Canturija \cite{17} established the following result. If
\begin{equation}
\mathfrak{a}:=\liminf_{t\rightarrow \infty }\int_{\tau (t)}^{t}p(s)ds>\frac{1}{e}\,,
\end{equation}%
then all solutions of the equation $(1.2)$ oscillate; if
\begin{equation}
\limsup_{t\rightarrow \infty }\int_{\tau (t)}^{t}p(s)ds<\frac{1}{e}\,,
\end{equation}%
then the equation $(1.2)$ has a non-oscillatory solution.

\vskip0.2cm Concerning the constants $1$ and $\frac{1}{e}$ which appear in
the conditions $(2.1)$, $(2.2)$ and $(2.3),$ in 2011, Berezansky and
Braverman \cite{2} established the following:

\begin{theo}[{\cite[Theorem 2.5]{2}}]
For any $\alpha\in (1/e,1)$ there exists a non-oscillatory equation
\begin{equation*}
x^{\prime }(t)+p(t)x(t-\tau )=0,\;\;\text{\ }\tau >0
\end{equation*}%
with $p(t)\geq 0$ such that
\begin{equation*}
\limsup_{t\rightarrow \infty }\int_{t-\tau }^{t}p(s)ds=\alpha.
\end{equation*}
\end{theo}

Also, in 2011, Braverman and Karpuz \cite{3} investigated equation
(1.2) in the case of a general argument ($\tau $ is not assumed
non-decreasing as in (2.1)) and proved that:

\begin{theo}[{\cite[Theorem 1]{3}}] There is
no constant $\mathit{K>0}$ such that%
\begin{equation}
\limsup_{t\rightarrow \infty }\int_{\tau (t)}^{t}p(s)ds>K
\end{equation}%
implies oscillation of equation $(1.2)$ for arbitrary
$($not necessarily non-decreasing$)$ argument  $\tau (t)\leq t.$
\end{theo}

\begin{rem}
Observe that, because of the condition $%
(2.3),$ the constant $K$ in $(2.4)$ makes sense for $K$ $>1/e.$
\end{rem}

Moreover in \cite{3} the above condition (2.1) was improved as
follows.

\begin{theo} [{\cite[Corollary 1]{3}}] Assume that%
\begin{equation}
B:=\limsup_{t\rightarrow \infty }\int_{\sigma (t)}^{t}p(s)\exp \left\{
\int_{\tau (s)}^{\sigma (t)}p(\xi )d\xi \right\} ds>1,
\end{equation}%
where  $\sigma (t)=\sup_{s\leq t}\tau (s)$, $t\geq t_{0.}$
Then all solutions of the equation $(1.2)$ oscillate.
\end{theo}

For equations with several arguments the following results have
been established.

\vskip0.1cm In 1982, Ladas and Stavroulakis \cite{22}, (see also in 1984, Arino,
Gyori and Jawhari [1]), studied the equation with several constant retarded
arguments of the form%
\begin{equation*}
x^{\prime }(t)+\sum_{i=1}^{m}p_{i}(t)x(t-\tau _{i})=0\text{ , \ }t\geq t_{0},%
\eqno(1.1^{\prime })
\end{equation*}%
under the assumption that $\liminf_{t\rightarrow \infty }
\int_{t-\tau_{i}/2}^{t}p(s)ds>0,$ $i=1,2,\dots,m,$ and proved that each one
of the following conditions
\begin{gather}
\liminf_{t\rightarrow \infty }\int_{t-\tau _{i}}^{t}p_{i}(s)ds>\mathrm{\frac{%
1}{e}}\ \ \text{for some} \ \ i,\ \ i=1,2,\dots,m,\\
\liminf_{t\rightarrow \infty }\int_{t-\tau
}^{t}\sum\limits_{i=1}^{m}p_{i}(s)ds>\mathrm{\frac{1}{e},}\ \ \text{where} \
\ \tau =\min \{\tau _{1,}\tau _{2,},\dots,\tau _{m}\},\\
\left[ \prod\limits_{i=1}^{m}\bigg( \sum\limits_{j=1}^{m}\liminf_{t%
\rightarrow \infty }\int_{t-\tau _{j}}^{t}p_{i}(s)ds\bigg) \right] ^{1/m}>%
\mathrm{\frac{1}{e}}\text{,}
\end{gather}%
or%
\begin{align}
\frac{1}{m}&\sum\limits_{i=1}^{m}\left( \liminf_{t\rightarrow \infty
}\int_{t-\tau _{i}}^{t}p_{i}(s)ds\right) +  \notag \\
&+\frac{2}{m}\sum\limits_{\substack{ i<j  \\ i,j=1}}^{m}\left[ \bigg( %
\liminf_{t\rightarrow \infty }\int_{t-\tau _{j}}^{t}p_{i}(s)ds\bigg) \bigg( %
\liminf_{t\rightarrow \infty }\int_{t-\tau _{i}}^{t}p_{i}(s)ds\bigg) \right]
^{\frac{1}{2}}>\mathrm{\frac{1}{e}}
\end{align}%
implies that all solutions of equation $(1.1^{\prime })$ oscillate. Later in
1996, Li \cite{24} proved that the same conclusion holds if
\begin{equation}
\liminf_{t\rightarrow \infty }\sum_{i=1}^{m}\int_{t-\tau _{i}}^{t}p_{i}(s)ds>%
\frac{1}{e}\text{.}
\end{equation}

\vskip0.1cm In 1984, Hunt and Yorke \cite{13}, considered the equation with
variable coefficients of the form:
\begin{equation*}
x^{\prime }(t)+\sum_{i=1}^{m}p_{i}(t)x(t-\tau _{i}(t))=0\text{ , \ }t\geq
t_{0},\eqno(1.1^{\prime \prime })
\end{equation*}%
under the assumption that there is a uniform upper bound $\tau _{0}$ on the $%
\tau _{i}$'$s$ and proved that if
\begin{equation*}
\liminf_{t\rightarrow \infty }\sum_{i=1}^{m}\tau _{i}(t)p_{i}(t)>\mathrm{%
\frac{1}{e}}
\end{equation*}%
Then all solutions of the equation $(1.1^{\prime \prime })$ oscillate.

\vskip0.1cm In 1984, Fukagai and Kusano \cite{8}, for the equation (1.1)
established the following theorem.

\begin{theo}[{\cite[Theorem 1$^{\prime }$ (i)]{8}}]
Consider equation $(1.1)$ and assume that there is a
continuous non-decreasing function $\tau ^{\ast }(t)$ such that
$\tau _{i}(t)\leq \tau ^{\ast }(t)\leq t$  for $t\geq t_{0}$,
 $1\leq i\leq m$. If%
\begin{equation}
\liminf_{t\rightarrow \infty }\int_{\tau ^{\ast
}(t)}^{t}\sum_{i=1}^{m}p_{i}(s)ds>\frac{1}{e}\text{,}
\end{equation}%
Then all solutions of the equation $(1.1)$ oscillate. If, on
the other hand, there exists a continuous non-decreasing  function
$\tau_{\ast }(t)$ such that $\tau_{\ast }(t)\leq \tau _{i}(t)$ %
for $t\geq t_{0}$, $1\leq i\leq m$, \ $\lim_{t\rightarrow
\infty }\tau _{\ast }(t)=\infty $ and %
\begin{equation*}
\int_{\tau _{\ast }(t)}^{t}\sum_{i=1}^{m}p_{i}(s)ds\leq \frac{1}{e}
\end{equation*}
for all sufficiently large $t$,
then the equation $(1.1)$ has a non-oscillatory solution.%
\end{theo}

\vskip0.1cm In 2000, Grammatikopoulos, Koplatadze and Stavroulakis \cite{10}
improved the above results as follows:

\begin{theo}[{\cite[Theorems 2.6]{10}}] Assume that the
functions $\tau _{i}$  are non-decreasing for all $i\in \left\{
1,\ldots ,m\right\} $,%
\begin{equation*}
\int_{0}^{\infty }\left\vert p_{i}(t)-p_{j}(t)\right\vert dt<+\infty ,\text{
\ \ \ }i,j=1,\dots,m
\end{equation*}%
and%
\begin{equation*}
\liminf_{t\rightarrow \infty }\int_{\tau _{i}(t)}^{t}p_{i}(s)ds=\beta
_{i}>0,\ \text{\ \ }i=1,\dots,m\text{.}
\end{equation*}%
If%
\begin{equation}
\sum_{i=1}^{m}\bigg( \liminf_{t\rightarrow \infty }\int_{\tau
_{i}(t)}^{t}p_{i}(s)ds\bigg) >\frac{1}{e},
\end{equation}%
Then all solutions of the equation $(1.1)$ oscillate.
\end{theo}

\section{New Oscillation Results}

Observe that all the above mentioned oscillation conditions
(2.6)--(2.12) involve $\liminf $ only and coincide with the condition (2.2)
in the case that $m=1.$ It is obvious that there is a gap between the
conditions $(2.1)$ and $(2.2)$ when the limit $\lim_{t{\rightarrow }{\infty }%
}\int_{\tau (t)}^{t}p(s)ds$ does not exist. Moreover, in view of Theorem
2.2, it is an interesting problem to investigate equation (1.1) with
non-monotone arguments and derive sufficient oscillation conditions,
involving $\limsup $ (as the condition (2.1) for the equation (1.2) with one
argument), which is the main objective in the following.

\begin{theo}
Assume that there exist non-decreasing functions $\sigma_{i}\in C([t_{0},+\infty))$
such that
\begin{equation}
\tau_{i}(t)\leq \sigma_{i}(t)\leq t\quad (i=1,\dots,m),
\end{equation}%
and
\begin{align}
\limsup_{t\rightarrow +\infty }& \prod_{j=1}^{m}\Bigg[\prod_{i=1}^{m}\int_{%
\sigma _{j}(t)}^{t}p_{i}(s)\exp \bigg(\int_{\tau _{i}(s)}^{\sigma
_{i}(t)}\sum_{i=1}^{m}p_{i}(\xi )\times  \notag \\
& \times \exp \bigg(\int_{\tau _{i}(\xi )}^{\xi }\sum_{i=1}^{m}p_{i}(u)du%
\bigg)d\xi \bigg)ds\Bigg]^{\frac{1}{m}}>\frac{1}{m^{m}}.
\end{align}%
Then all solutions of the equation $(1.1)$ oscillate.
\end{theo}

\begin{proof} Assume, for the sake of contradiction, that the
retarded equation $(1.1)$ admits a non-oscillatory solution $x(t).$ Since $%
-x(t)$ is also a solution to $(1.1)$, we can confine ourselves only to the
case that $x(t)$ is an eventually positive solution of the equation $(1.1).$
Then there exists $t_{1}>t_{0}$ such  that $x(t)$, $x(\tau _{i}(t)),$\ $%
x(\sigma _{i}(t))>0$ $(i=1,2\dots,m)$ for all $t\geq t_{1}.$ Therefore,
from equation $(1.1)$ it follows that $x^{\prime }(t)\leq 0$ for all $t\geq
t_{1}$ and consequently $x(t)$\ is non-increasing. From $(1.1)$ dividing by $%
x(t)$ and integrating from $s$ to $t$ for sufficiently large $t$ and $s,$ we
have
\begin{equation}
x(s)=x(t)\exp \bigg\{\int_{s}^{t}\sum_{i=1}^{m}p_{i}(\xi )\frac{x(\tau
_{i}(\xi ))}{x(\xi )}\,d\xi \bigg\}\quad \text{for}\quad t\geq s.
\end{equation}%
Integrating (1.1) from $\sigma _{j}(t)$ to $t$, for sufficiently large $t,$
we have%
\begin{equation}
x(\sigma _{j}(t))\geq \sum_{i=1}^{m}\int_{\sigma _{j}(t)}^{t}x(\tau
_{i}(s))\,p_{i}(s)\,ds.
\end{equation}%
On the other hand from (3.3), taking into account (3.1) and the fact that
the function $x(t)$\ is non-increasing, for sufficiently large $t,$ we obtain%
\begin{equation}
\frac{x(\tau _{i}(t))}{x(t)}\!=\!\exp \bigg\{\int_{\tau
_{i}(t)}^{t}\sum_{i=1}^{m}p_{i}(\xi )\frac{x(\tau _{i}(\xi ))}{x(\xi )}d\xi %
\bigg\}\geq \exp \bigg\{\int_{\tau _{i}(t)}^{t}\sum_{i=1}^{m}p_{i}(\xi )d\xi %
\bigg\}
\end{equation}%
and
\begin{equation}
x(\tau _{i}(s))\!\geq \!x(\sigma _{i}(t))\exp \Bigg\{\!\int_{\tau
_{i}(s)}^{\sigma _{i}(t)}\sum_{i=1}^{m}p_{i}(\xi )\exp \bigg(\!\int_{\tau
_{i}(\xi )}^{\xi }\!\sum_{i=1}^{m}p_{i}(u)\,du\!\bigg)d\xi \Bigg\}.
\end{equation}%
Combining the last three inequalities (3.4), (3.5) and (3.6), and using
the arithmetic mean-geometric mean inequality, we get%
\begin{align*}
x(\sigma _{j}(t))& \geq \sum_{i=1}^{m}\int_{\sigma _{j}(t)}^{t}x(\sigma
_{i}(t))\times \\
& \times \exp \Bigg\{\int_{\tau _{i}(s)}^{\sigma
_{i}(t)}\sum_{i=1}^{m}p_{i}(\xi )\exp \bigg(\int_{\tau _{i}(\xi )}^{\xi}
\sum_{i=1}^{m}p_{i}(u)du\bigg)d\xi \Bigg\}p_{i}(s)\,ds= \\
&= \sum_{i=1}^{m}x(\sigma_{i}(t))\times \\
&\times\int_{\sigma _{j}(t)}^{t} \exp \Bigg\{\int_{\tau _{i}(s)}^{\sigma
_{i}(t)}\sum_{i=1}^{m}p_{i}(\xi )\exp \bigg(\int_{\tau _{i}(\xi )}^{\xi
}\sum_{i=1}^{m}p_{i}(u)du\bigg)d\xi \Bigg\}p_{i}(s)\,ds\geq \\
& \geq m\bigg[\prod_{i=1}^{m}x(\sigma _{i}(t))\bigg]^{\frac{1}{m}}\times %
\Bigg[\prod_{i=1}^{m}\int_{\sigma _{j}(t)}^{t}p_{i}(s)\times \\
& \times \exp \bigg\{\int_{\tau _{i}(s)}^{\sigma
_{i}(t)}\sum_{i=1}^{m}p_{i}(\xi )\exp \bigg(\int_{\tau _{i}(\xi )}^{\xi
}\sum_{i=1}^{m}p_{i}(u)\,du\bigg)d\xi \bigg\}ds\Bigg]^{\frac{1}{m}},\text{ }%
j=1,\dots ,m,
\end{align*}%
and taking the product on both sides of the last inequality, we find
\begin{align*}
w(t)& \geq m^{m}w(t)\prod_{j=1}^{m}\Bigg[\prod_{i=1}^{m}\int_{\sigma
_{j}(t)}^{t}\exp \bigg\{\int_{\tau _{i}(s)}^{\sigma
_{i}(t)}\sum_{i=1}^{m}p_{i}(\xi )\times \\
& \times \exp \bigg(\int_{\tau _{i}(\xi )}^{\xi }\sum_{i=1}^{m}p_{i}(u)\,du%
\bigg)d\xi \bigg\}ds\Bigg]^{\frac{1}{m}},
\end{align*}%
where $w(t)=\prod\limits_{j=1}^{m}x(\sigma _{j}(t))$. That is,
\begin{align*}
\limsup_{t\rightarrow +\infty }& \prod_{j=1}^{m}\Bigg[\prod_{i=1}^{m}\int_{%
\sigma _{j}(t)}^{t}p_{i}(s)\exp \bigg(\int_{\tau _{i}(s)}^{\sigma
_{i}(t)}\sum_{i=1}^{m}p_{i}(\xi )\times \\
& \times \exp \bigg(\int_{\tau _{i}(\xi )}^{\xi }\sum_{i=1}^{m}p_{i}(u)du%
\bigg)d\xi \bigg)ds\Bigg]^{\frac{1}{m}}\leq \frac{1}{m^{m}}\,,
\end{align*}%
which contradicts (3.2). The proof of the theorem is complete.
\end{proof}

For the next theorem we need the following lemma.

\begin{lem}
Let
\begin{equation}
\liminf_{t\rightarrow +\infty }\int_{\tau
_{i}(t)}^{t}p_{i}(s)\,ds=p_{i}>0\quad (i=1,\dots ,m).
\end{equation}%
and $x(t)$ be an eventually positive solution of equation $(1.1)$. Then
\begin{equation}
\liminf_{t\rightarrow +\infty }\frac{x(\tau _{i}(t))}{x(t)}\geq \lambda_{i}^{\ast },
\end{equation}%
where $\lambda _{i}^{\ast }$ is the smallest root of the equation%
\begin{equation}
e^{p_{i}\lambda }=\lambda .  \tag{$3.8'$}
\end{equation}
\end{lem}

\begin{proof}
Let $\varepsilon \in (0,p_{i})$. First we show that
\begin{equation}
\liminf_{t\rightarrow +\infty }\frac{x(\tau _{i}(t))}{x(t)}\geq \lambda
_{i}^{\ast }(\varepsilon ),
\end{equation}%
where $\lambda _{i}^{\ast }(\varepsilon )$ is the smallest root of the
equation%
\begin{equation*}
e^{(p_{i}-\varepsilon )\lambda }=\lambda .
\end{equation*}%
By (3.7), for any $\varepsilon >0$, there exist $t_{\varepsilon }\in R_{+}$
such that
\begin{equation}
\int_{\tau _{i}(t)}^{t}p_{i}(s)\,ds\geq p_{i}-\varepsilon \quad \text{for}%
\quad t\geq t_{\varepsilon }.
\end{equation}%
Assume, for the sake of contradiction, that (3.9) is not correct. Then,
there exists $\varepsilon _{0}>0$ such that%
\begin{equation}
\frac{e^{(p_{i}-\varepsilon )\gamma _{i}}}{\gamma _{i}}\geq 1+\varepsilon_{0},
\end{equation}%
where%
\begin{equation}
\gamma _{i}=\liminf_{t\rightarrow +\infty }\frac{x(\tau _{i}(t))}{x(t)}%
<\lambda _{i}^{\ast }(\varepsilon ).
\end{equation}%
On the other hand, for any $\delta >0$ there exists $t_{\delta }$ such that
\begin{equation*}
\frac{x(\tau _{i}(t))}{x(t)}\geq \gamma _{i}-\delta \quad \text{for}\quad
t\geq t_{\delta }.
\end{equation*}%
Dividing $(1.1)$ by $x(t)$ and integrating from $\tau _{i}(t)$ to $t$ \ for
sufficiently large $t,$ and taking into account $(3.10)$ and the last
inequality, we have
\begin{equation*}
\frac{x(\tau _{i}(t))}{x(t)}\geq \exp \bigg(\int_{\tau _{i}(t)}^{t}\frac{%
x(\tau _{i}(s))}{x(s)}\,p(s)\,ds\bigg)\geq e^{(p_{i}-\varepsilon )(\gamma
_{i}-\delta )}\quad \text{for large \ }t.
\end{equation*}%
Therefore,
\begin{equation*}
\gamma _{i}=\liminf_{t\rightarrow +\infty }\frac{x(\tau _{i}(t))}{x(t)}\geq
e^{(p_{i}-\varepsilon )(\gamma _{i}-\delta )}
\end{equation*}%
which implies
\begin{equation*}
\gamma _{i}\geq e^{\gamma _{i}(p_{i}-\varepsilon )}.
\end{equation*}%
In view of (3.11),\ this is a contradiction and therefore (3.9) is true.
Since $\lambda _{i}^{\ast }(\varepsilon )\rightarrow \lambda _{i}^{\ast }$
as $\varepsilon \rightarrow 0$, (3.9) implies (3.8) and the proof of the
lemma is complete.
\end{proof}

Using this lemma, as in Theorem 3.1, we can prove the following.

\begin{theo}
Let $(3.1)$ be fulfilled and for some non-decreasing functions $\sigma _{i}\in C([t_{0},+\infty ))$
$(i=1,\dots ,m)$
\begin{align}
\limsup_{\varepsilon \rightarrow 0+}& \Bigg(\limsup_{t\rightarrow +\infty
}\prod_{j=1}^{m}\bigg(\prod_{i=1}^{m}\int_{\sigma _{j}(t)}^{t}p_{i}(s)\times
\notag \\
& \times \exp \bigg(\int_{\tau _{i}(s)}^{\sigma
_{i}(t)}\sum_{i=1}^{m}(\lambda _{i}^{\ast }-\varepsilon )p_{i}(\xi )d\xi %
\bigg)ds\bigg)^{\frac{1}{m}}\Bigg)>\frac{1}{m^{m}}\,,
\end{align}%
where $\lambda _{i}^{\ast }$ $(i=1,\dots ,m)$ is the
smallest root of the equation $(3.8^{\prime })$. Then all solutions of the
equation $(1.1)$ oscillate.
\end{theo}

When $m=1,$ that is in the case of the equation (1.2) with one
argument, from Theorems 3.1 and 3.2 we have, respectively, the following corollaries.

\begin{coro}
Let
\begin{equation}
C:=\limsup_{t\rightarrow +\infty }\int_{\sigma (t)}^{t}p(s)\exp \Bigg(%
\int_{\tau (s)}^{\sigma (t)}p(\xi )\exp \bigg(\int_{\tau (\xi )}^{\xi }p(u)du%
\bigg)d\xi \Bigg)ds>1.
\end{equation}%
Then all solutions of the equation $(1.2)$ oscillate.
\end{coro}

\begin{coro}
Let
\begin{equation}
\limsup_{\varepsilon \rightarrow 0}\Bigg(\limsup_{t\rightarrow +\infty
}\int_{\sigma (t)}^{t}p(s)\exp \bigg(\int_{\tau (s)}^{\sigma (t)}(\lambda
^{\ast }-\varepsilon )p(\xi )d\xi \bigg)ds\Bigg)>1\,.
\end{equation}%
Then all solutions of the equation $(1.2)$ oscillate.
\end{coro}

In the case of monotone arguments we have the following.

\begin{theo}
Let $\tau _{i}$ be non-decreasing
functions and
\begin{align*}
\limsup_{t\rightarrow +\infty }& \prod_{j=1}^{m}\Bigg[\prod_{i=1}^{m}\int_{%
\tau _{j}(t)}^{t}p_{i}(s)\exp \bigg(\int_{\tau _{i}(s)}^{\tau
_{i}(t)}\sum_{i=1}^{m}p_{i}(\xi )\times \\
& \times \exp \bigg(\int_{\tau _{i}(\xi )}^{\xi }\sum_{i=1}^{m}p_{i}(u)du%
\bigg)d\xi \bigg)ds\Bigg]^{\frac{1}{m}}>\frac{1}{m^{m}}.\,
\end{align*}%
or
\begin{align*}
\limsup_{\varepsilon \rightarrow 0+}& \Bigg(\limsup_{t\rightarrow +\infty
}\prod_{j=1}^{m}\bigg(\prod_{i=1}^{m}\int_{\tau _{j}(t)}^{t}p_{i}(s)\times \\
& \times \exp \bigg(\int_{\tau _{i}(s)}^{\tau _{i}(t)}\sum_{i=1}^{m}(\lambda
_{i}^{\ast }-\varepsilon )p_{i}(\xi )d\xi \bigg)ds\bigg)^{\frac{1}{m}}\Bigg)>%
\frac{1}{m^{m}}\,,
\end{align*}%
where $\lambda _{i}^{\ast }$ $(i=1,\dots ,m)$ is the
smallest root of the equation $(3.8^{\prime })$. Then all solutions of the
equation $(1.1)$ oscillate.
\end{theo}

\begin{coro}
Let $\tau _{i}$ be non-decreasing functions and
\begin{equation}
\limsup_{t\rightarrow +\infty }\prod_{j=1}^{m}\bigg(\prod_{i=1}^{m}\int_{%
\tau _{j}(t)}^{t}p_{i}(s)ds\bigg)^{\frac{1}{m}}>\frac{1}{m^{m}}.
\end{equation}%
Then all solutions of the equation $(1.1)$ oscillate.
\end{coro}

\begin{coro}
Let $\tau _{i}$ be non-decreasing
functions, $p/ _{i}(t)\geq p(t)$ $(i=1,\dots ,m)$ and
\begin{equation}
\limsup_{t\rightarrow +\infty }\prod_{j=1}^{m}\int_{\tau
_{j}(t)}^{t}p(s)\,ds>\frac{1}{m^{m}},
\end{equation}%
Then all solutions of the equation $(1.1)$ oscillate.
\end{coro}

\begin{coro}
Let $\tau _{i}$ be non-decreasing
functions, $p_{i}(t)\geq p=$\textrm{const} and
\begin{equation}
p^{m}\limsup_{t\rightarrow +\infty }\prod_{i=1}^{m}(t-\tau _{i}(t))>\frac{1}{%
m^{m}}.
\end{equation}%
Then all solutions of the equation $(1.1)$ oscillate.
\end{coro}

\section{Remarks and Examples}

\begin{rem}
It should be pointed out that the conditions
$(3.2)$ and $(3.13)$ of Theorems $3.1$ and $3.2$ present for the first time
sufficient conditions (in terms of $\limsup)$ for the oscillation of all
solutions to the equation $(1.1)$ with several non-monotone arguments. They
are also independent and essentially improve all\ the related oscillation
conditions in the literature. Even in the case where $m=1,$ the improvement
is substantial.
\end{rem}

\begin{rem}
Observe that when $m=1,$ the above conditions
$(3.16)$ and $(3.17)$ reduces to the $($classical$)$ condition $(2.1)$.
\end{rem}

The following examples illustrate the significance of our results.

\begin{exmp}
Let $a_{1},\delta \in (0,+\infty )$ and
consider the sequences $\{a_{k}\}_{n=1}^{+\infty }$, where $%
a_{k+1}=3+2\delta +a_{k}$ $(k=1,2,\dots )$. Choose $\alpha \in (0,1)$ and $%
\varepsilon >0$ such that
\begin{equation}
\frac{\ln (1+e)}{e-\varepsilon }<\alpha <\ln 2.
\end{equation}%
Consider equation $(1.2)$, where $\tau (t)=t-1$ and
\begin{equation}
p(t)=\begin{cases}
\;\dfrac{1}{e}\qquad \qquad \quad \text{for}\quad a_{k}\leq t\leq a_{k}+1, \\[2mm]
\dfrac{\alpha e-1}{\delta e}\,t+\dfrac{\delta +(1+a_{k})(1-\alpha e)}{\delta
e} \\[1mm]
\qquad \qquad \qquad \text{for}\quad a_{k}+1\leq t\leq a_{k}+\delta +1, \\[2mm]
\;\alpha \quad \qquad \qquad \text{for}\quad a_{k}+\delta +1\leq t\leq
a_{k+1}-\delta , \\[2mm]
\dfrac{1-\alpha e}{\delta e}\,t+\dfrac{\alpha \delta e+(a_{k}+\delta
+3)(\alpha e-1)}{\delta e} \\[1mm]
\qquad \quad \qquad \;\;\;\text{for}\quad a_{k+1}-\delta \leq t\leq a_{k+1}
\\[2mm]
\qquad \qquad \qquad \qquad \qquad \qquad (k=1,2,\dots ).%
\end{cases}%
\end{equation}

It is obvious that
\begin{equation}
\liminf_{t\rightarrow +\infty }\int_{\tau (t)}^{t}p(s)\,ds=\frac{1}{e}\quad
\text{and}\quad \lambda ^{\ast }=e
\end{equation}%
(see Lemma 2.1). By (4.2) and (4.1)
\begin{align}
\limsup_{t\rightarrow +\infty }& \int_{\tau (t)}^{t}p(s)\exp \bigg\{%
\int_{\tau (s)}^{\tau (t)}p(s)\,ds\bigg\}\leq  \notag \\
& \leq \limsup_{t\rightarrow +\infty }\int_{t-1}^{t}\alpha \,e^{\alpha
(t-s)}\,ds=e^{\alpha }-1<1. 
\end{align}%
On the other hand by (4.1)--(4.3) we have
\begin{align}
\limsup_{t\rightarrow +\infty }& \int_{\tau (t)}^{t}p(s)\exp \bigg\{%
(e-\varepsilon )\int_{\tau (s)}^{\tau (t)}p(\xi )\,d\xi \bigg\}ds=  \notag \\
& =\limsup_{t\rightarrow +\infty }\int_{t-1}^{t}\alpha \exp \bigg\{%
(e-\varepsilon )^{\alpha (t-s)}\bigg\}ds=  \notag \\
&=\frac{1}{e-\varepsilon }\Big(e^{\alpha (e-\varepsilon )}-1\Big)>\frac{e}{%
e-\varepsilon }>1.  
\end{align}%
Therefore, according to (4.4) the condition (2.5) is not fulfilled. On the
other hand by (4.5) the condition (3.13) holds a.e. and therefore all
solutions of equation (1.2) oscillate.
\end{exmp}

\begin{exmp}[{cf. \cite{3}}] We consider a generalisation of an
example presented in \cite{3}, where the equation%
\begin{equation*}
x^{\prime }(t)+\frac{1}{e}x(\tau (t))=0,\text{ }t\geq 0,
\end{equation*}%
with the retarded argument%
\begin{equation*}
\tau (t):=\begin{cases}
t-1,\,&t\in [3n,3n+1], \\
-3t+(12n+3),\,&t\in [3n+1,3n+2], \\
5t-(12n+13),\,&t\in [3n+2,3n+3],%
\end{cases}%
\end{equation*}%
was studied. Here we discuss the more general equation
\begin{equation}
x^{\prime }(t)+px(\tau (t))=0,\text{ }t\geq 0,\text{ }p>0,
\end{equation}%
and illustrate how our methodology can be utilized to prove the existence of
oscillatory solutions for some range of the parameter $p$. In this case, as
in \cite{3}, one may choose the function%
\begin{equation*}
\sigma (t)=\begin{cases}
t-1,\,&t\in[3n,3n+1], \\
3n,\,&t\in [3n+1,3n+2.6], \\
5t-(12n+13),\,&t\in [3n+2.6,3n+3].%
\end{cases}%
\end{equation*}%
Now note that, since $\tau (t)\leq t-1$,
\begin{equation*}
\int_{\tau (t)}^{t}pdu\geq \int_{t-1}^{t}pdu=p.
\end{equation*}%
The choice as in \cite{3} of $t_{n}=3n+3$ gives
\begin{align*}
C &=\limsup_{t\rightarrow +\infty }\int_{\sigma (t)}^{t}p\exp \bigg(%
\int_{\tau (s)}^{\sigma (t)}p\exp \bigg(\int_{\tau (\xi )}^{\xi }pdu\bigg)%
d\xi \bigg)ds \\
&\geq \lim_{n\rightarrow +\infty }\int_{3n+2}^{3n+3}p\exp \bigg(%
\int_{5s-(12n+13)}^{3n+2}p\exp (p)d\xi \bigg)ds=\frac{1}{5}\,\left( {{e}%
^{5\,p{{e}^{p}}}}-1\right) {{e}^{-p}}.
\end{align*}%
The inequality
\begin{equation*}
\frac{1}{5}\,\left( {{e}^{5\,p{{e}^{p}}}}-1\right) {{e}^{-p}}>1
\end{equation*}%
is satisfied for (the numbers that follow are rounded to the third decimal
place unless exact) $p\in \lbrack 0.303,0.358].$ Thus, for $p\in \lbrack
0.303,0.358]$ the condition (3.14) of Corollary 2.1 is satisfied and
therefore all solutions to the above equation (4.6) oscillate. Observe,
however, that when $p\in \lbrack 0.303,0.358]$ in (4.6), we find
\begin{gather*}
A=\limsup_{t\rightarrow \infty }\int_{\sigma (t)}^{t}pds=p\cdot (2.6)<1,\\
\mathfrak{a}:=\liminf_{t\rightarrow \infty }\int_{\tau (t)}^{t}pds=p<\frac{1}{e}
\end{gather*}%
and%
\begin{align*}
\int_{\sigma (3n+3)}^{3n+3}p\exp \left\{ \int_{\tau (s)}^{\sigma
(3n+3)}pd\xi \right\} ds&=\int_{3n+2}^{3n+3}p\exp \left\{
\int_{5s-(12n+13)}^{3n+2}pd\xi \right\} ds= \\
&=\frac{1}{5}(e^{5p}-1)<1.
\end{align*}%
That is, none of the known oscillation conditions (2.1), (2.2) (and also the
conditions (2.6)--(2.11)) and (2.5) is satisfied.
\end{exmp}

\begin{rem}
It is obvious that if for some $i_{0}\in\{1,\dots ,m\}$ all solutions of the equation
\begin{equation*}
x^{\prime }(t)+p_{i_{{}_{0}}}(t)\,x(\tau _{i_{{}_{0}}}(t))=0
\end{equation*}%
oscillate, then all solutions of the equation (1.1) also oscillate. 
\end{rem}

\begin{exmp}
Let $p,\Delta _{1},\Delta _{2}\in (0,+\infty )$ and
consider the sequences $\big\{t_{k}\big\}_{k=1}^{\infty }$ such that $%
t_{k}\uparrow +\infty $ for $k\uparrow +\infty $, $t_{k}+2\Delta <t_{k+1}$ $%
(k=1,2,\dots )$, where $\Delta =\max \{\Delta _{i},\;i=1,2\}$. Choose $p$, $%
\Delta _{1}$ and $\Delta _{2}$ such that
\begin{equation}
p^{2}\,\Delta _{1}\,\Delta _{2}>\frac{1}{4}
\end{equation}%
and
\begin{equation}
p\,\Delta _{i}<1\quad (i=1,2).
\end{equation}%
Let $p(t)=p$ for $t\in \lbrack t_{k},t_{k}+\Delta ]$ $(k=1,2,\dots )$ and $%
p(t)=0$ for $t\in R_{+}\backslash \underset{k=1}{\overset{\infty }{U}}%
[t_{k},t_{k}+\Delta ]$.

According to (4.7) it is obvious that the condition (3.18) is fulfilled,
where $m=2$ and $\tau _{i}(t)=t-\Delta _{i}$ $(i=1,2)$, a.e. and therefore
all solutions to equation (1.1) are oscillatory. However, for the equations
\begin{equation*}
x^{\prime }(t)+p(t)\,x(t-\Delta _{i})=0\quad (i=1,2)
\end{equation*}%
by (4.8), we have
\begin{equation*}
\limsup_{t\rightarrow +\infty }\int_{t-\Delta _{i}}^{t}p(s)\,ds<1\quad
(i=1,2)
\end{equation*}%
and
\begin{equation*}
\liminf_{t\rightarrow +\infty }\int_{t-\Delta _{i}}^{t}p(s)\,ds=0\quad
(i=1,2).
\end{equation*}%
\end{exmp}

\begin{rem}
In the above mentioned Example 4.3, by a
solution, we mean an absolutely continuous function which satisfies the
corresponding equation almost everywhere.
\end{rem}

\begin{exmp}
Consider the equation%
\begin{equation}
x^{\prime }(t)+p_{1}x(\tau _{1}(t))+p_{2}x(\tau _{2}(t))=0,\,t\geq
0,\,p_{1},p_{2}>0,
\end{equation}%
where%
\begin{align*}
\tau _{1}(t)& =\begin{cases}
t-1, & t\in [3n,3n+1], \\
-3t+(12n+3), & t\in [3n+1,3n+2], \\
5t-(12n+13), & t\in [3n+2,3n+3],%
\end{cases}\\
\tau _{2}(t)& =\begin{cases}
t-2, & t\in [3n,3n+1], \\
-t+6n, & t\in [3n+1,3n+2], \\
3t-(6n+8), & t\in [3n+2,3n+3].%
\end{cases}%
\end{align*}%
We can take
\begin{align*}
\sigma _{1}(t)& =\begin{cases}
t-1, & t\in [3n,3n+1], \\
3n, & t\in [3n+1,3n+2.6], \\
5t-(12n+13), & t\in [3n+2.6,3n+3],%
\end{cases}
\\
\sigma _{2}(t)& =\begin{cases}
t-2, & t\in [3n,3n+1], \\
3n-1, & t\in [3n+1,3n+2.\bar{3}], \\
3t-(6n+8), & t\in [3n+2.\bar{3},3n+3].%
\end{cases}%
\end{align*}%
Note that, since $\tau _{1}(t)\leq t-1$ and $\tau _{2}(t)\leq t-2$, we have
\begin{equation*}
\int_{\tau _{1}(t)}^{t}du\geq \int_{t-1}^{t}du=1,\quad \int_{\tau
_{2}(t)}^{t}du\geq \int_{t-2}^{t}du=2.
\end{equation*}%
Set $P=p_{1}\exp (p_{1}+p_{2})+p_{2}\exp (2p_{1}+2p_{2}).$ The choice of $%
t_{n}=3n+3$ gives%
\begin{equation*}
\limsup_{t\rightarrow +\infty }\prod_{j=1}^{2}\bigg(\prod_{i=1}^{2}\int_{%
\sigma _{j}(t)}^{t}p_{i}\exp \bigg(\int_{\tau _{i}(s)}^{\sigma
_{i}(t)}\sum_{i=1}^{2}p_{i}\exp \bigg(\int_{\tau _{i}(\xi )}^{\xi
}(p_{1}+p_{2})du\bigg)d\xi \bigg)ds\bigg)^{\frac{1}{2}}
\end{equation*}%
\begin{equation*}
\geq \lim_{n\rightarrow +\infty }\prod_{j=1}^{2}\bigg(\prod_{i=1}^{2}\int_{%
\sigma _{j}(3n+3)}^{3n+3}p_{i}\exp \bigg(\int_{\tau _{i}(s)}^{\sigma
_{i}(3n+3)}\sum_{i=1}^{2}p_{i}\exp \bigg(\int_{\tau _{i}(\xi )}^{\xi
}(p_{1}+p_{2})du\bigg)d\xi \bigg)ds\bigg)^{\frac{1}{2}}
\end{equation*}

\begin{equation*}
\geq \lim_{n\rightarrow +\infty }\prod_{j=1}^{2}\bigg(\int_{\sigma
_{j}(3n+3)}^{3n+3}p_{1}\exp \bigg(\int_{\tau _{1}(s)}^{3n+2}Pd\xi \bigg)ds%
\bigg)^{\frac{1}{2}}\times \bigg(\int_{\sigma _{j}(3n+3)}^{3n+3}p_{2}\exp %
\bigg(\int_{\tau _{2}(s)}^{3n+1}Pd\xi \bigg)ds\bigg)^{\frac{1}{2}}
\end{equation*}%
\begin{eqnarray*}
&=&\lim_{n\rightarrow +\infty }\bigg(\int_{3n+2}^{3n+3}p_{1}\exp \bigg(%
\int_{\tau _{1}(s)}^{3n+2}Pd\xi \bigg)ds\bigg)^{\frac{1}{2}}\times \bigg(%
\int_{3n+2}^{3n+3}p_{2}\exp \bigg(\int_{\tau _{2}(s)}^{3n+1}Pd\xi \bigg)ds%
\bigg)^{\frac{1}{2}} \\
&&\times \bigg(\int_{3n+1}^{3n+3}p_{1}\exp \bigg(\int_{\tau
_{1}(s)}^{3n+2}Pd\xi \bigg)ds\bigg)^{\frac{1}{2}}\times \bigg(%
\int_{3n+1}^{3n+3}p_{2}\exp \bigg(\int_{\tau _{2}(s)}^{3n+1}Pd\xi \bigg)ds%
\bigg)^{\frac{1}{2}}
\end{eqnarray*}%
\begin{eqnarray*}
&=&\lim_{n\rightarrow +\infty }\bigg(\int_{3n+2}^{3n+3}p_{1}\exp \bigg(%
\int_{5s-(12n+13)}^{3n+2}Pd\xi \bigg)ds\bigg)^{\frac{1}{2}}\times \bigg(%
\int_{3n+2}^{3n+3}p_{2}\exp \bigg(\int_{3s-(6n+8)}^{3n+1}Pd\xi \bigg)ds\bigg)%
^{\frac{1}{2}} \\
&&\times \bigg(\int_{3n+1}^{3n+2}p_{1}\exp \bigg(\int_{-3s+(12n+3)}^{3n+2}Pd%
\xi \bigg)ds+\int_{3n+2}^{3n+3}p_{1}\exp \bigg(\int_{5s-(12n+13)}^{3n+2}Pd%
\xi \bigg)ds\bigg)^{\frac{1}{2}} \\
&&\times \bigg(\int_{3n+1}^{3n+2}p_{2}\exp \bigg(\int_{-s+6n}^{3n+1}Pd\xi %
\bigg)ds+\int_{3n+2}^{3n+3}p_{2}\exp \bigg(\int_{3s-(6n+8)}^{3n+1}Pd\xi %
\bigg)ds\bigg)^{\frac{1}{2}} \\
&=:&D(p_{1},p_{2}).
\end{eqnarray*}
Let $p_{1}=0.1$ then, by direct computation, we get
\begin{equation*}
D>\frac{1}{4},
\end{equation*}%
if $p_{2}\geq 0.158$. That\ is, when $p_{1}=0.1$ and $p_{2}\geq 0.158$ in
equation (4.9), the condition (3.1) of Theorem 3.1 is satisfied and
therefore all solutions to this equation oscillate.

Note that since the delays are not monotone, Theorem 2.5 cannot be applied
to this example. We now compare our result with Theorem 2.4. Note that
\begin{equation*}
\tau _{1}(t),\tau _{2}(t)\leq \sigma _{1}(t),\ \text{for every}\ t>0.
\end{equation*}%
The choice $p_{1}=0.1$, $p_{2}=0.158$ gives
\begin{equation*}
\liminf_{t\rightarrow \infty }\int_{\sigma
_{1}(t)}^{t}(p_{1}+p_{2}\,)ds=p_{1}+p_{2}=0.258<\frac{1}{e}.
\end{equation*}
\end{exmp}

\section*{Acknowledgement} 

This paper was partially written during a visit of G. Infante to the Department of Mathematics of the
University of Ioannina. G. Infante is grateful to the people of the
aforementioned Department for their kind and warm hospitality

\section*{$^*$ Acknowledgement} 

The work was supported by the Sh.
Rustaveli National Science Foundation. Grant No. 31/09.

\end{document}